\documentclass[12pt]{article}
\usepackage{amsmath,amsthm,amsfonts,amssymb}
\usepackage[pdftex,pdfborder={0 0 0}]{hyperref}
\usepackage{fullpage}
\usepackage{multirow}
\usepackage{bbm}
\usepackage[noblocks]{authblk}
\usepackage{MnSymbol}

\newtheorem{theorem}{Theorem}[section]
\newtheorem{lemma}[theorem]{Lemma}
\newtheorem{proposition}[theorem]{Proposition}
\newtheorem{corollary}[theorem]{Corollary}
\newtheorem{conjecture}[theorem]{Conjecture}

\theoremstyle{definition}
\newtheorem{definition}[theorem]{Definition}
\newtheorem{example}[theorem]{Example}

\newlength{\Oldarrayrulewidth}

\newcommand{\C}{\mathbb{C}}
\newcommand{\F}{\mathbb{F}}
\newcommand{\N}{\mathbb{N}}

\newcommand{\bz}{\textup{\textbf{0}}}

\makeatletter
\newcommand{\doublewidetilde}[1]{{%
  \mathpalette\double@widetilde{#1}%
}}
\newcommand{\double@widetilde}[2]{%
  \sbox\z@{$\m@th#1\widetilde{#2}$}%
  \ht\z@=.9\ht\z@
  \widetilde{\box\z@}%
}
\makeatother

\def\m@th{\mathsurround=0pt}
\def\sm#1{\null\,\vcenter{\baselineskip9pt\lineskip.23ex\m@th
    \ialign{\hfil$\scriptstyle##$\hfil&&\ \hfil$\scriptstyle##$\hfil\crcr
    \mathstrut\crcr\noalign{\kern-\baselineskip}
    #1\crcr\mathstrut\crcr\noalign{\kern-\baselineskip}}}\,}
\def\smnp#1{\null\,\vcenter{\baselineskip9pt\lineskip.23ex\m@th
    \ialign{\hfil$\scriptstyle##$\hfil&&\ \ \hfil$\scriptstyle##$\hfil\crcr
    \mathstrut\crcr\noalign{\kern-\baselineskip}
    #1\crcr\mathstrut\crcr\noalign{\kern-\baselineskip}}}\,}

\begin{document}

\title{Building matrices with prescribed size and number of invertible submatrices}
\author[1]{Edward S.\ T.\ Fan\thanks{edward\_fan@math.brown.edu}}
\author[2]{Tony W.\ H.\ Wong\thanks{wong@kutztown.edu}}
\affil[1]{Department of Mathematics, Brown University}
\affil[2]{Department of Mathematics, Kutztown University of Pennsylvania}
\date{\today}

\maketitle

\begin{abstract}
Given an ordered triple of positive integers $(n,r,b)$, where $1\leq b\leq\binom{n}{r}$, does there exist a matrix of size $r\times n$ with exactly $b$ invertible submatrices of size $r\times r$? Such a matrix is called an $(n,r,b)$-matrix. This question is a stronger version of an open problem in matroid theory raised by Dominic Welsh. In this paper, we prove that an $(n,r,b)$-matrix exists when the corank satisfies $n-r\leq3$, unless $(n,r,b)=(6,3,11)$. Furthermore, we show that an $(n,r,b)$-matrix exists when the rank $r$ is large relative to the corank $n-r$.
\end{abstract}

\section{Introduction}

Throughout the paper, let $n$ and $r$ denote two positive integers such that $r\leq n$. Let $A$ be a matrix of size $r\times n$ over a field $\F$ with full row rank. Let $b$ be the number of invertible $r\times r$ submatrices of $A$. What are the possible values of $b$? From basic linear algebra and simple counting, we know that $b$ must be between $1$ and $\binom{n}{r}$ inclusively.

Here is a more interesting question: is every such $b$ attainable? In other words, if we are given an ordered triple of positive integers $(n,r,b)$ such that $1\leq b\leq\binom{n}{r}$, which implicitly implies $r\leq n$, can we always build a matrix $A$ of size $r\times n$ over a field, such that the number of invertible $r\times r$ submatrices is exactly $b$? We denote such a matrix as an \textit{$(n,r,b)$-matrix}.

It turns out that the answer to the aforementioned question is known to be negative, thanks to the work by Anna de Mier on matroid theory. Here is a brief introduction on matroids.

\begin{definition} A \textit{matroid} $\mathcal{M}=(X,\mathcal{I})$ is a combinatorial structure defined on a finite ground set $X$ of $n$ elements, together with a family $\mathcal{I}$ of subsets of $X$ called \textit{independent sets}, satisfying the following three properties.
\begin{enumerate}
\item $\emptyset\in\mathcal{I}$.
\item (\textit{Hereditary property}) If $I\in\mathcal{I}$ and $J\subseteq I$, then $J\in\mathcal{I}$.
\item (\textit{Augmentation property}) If $I,J\in\mathcal{I}$ and $|J|<|I|$, then there exists $x\in I\backslash J$ such that $J\cup\{x\}\in\mathcal{I}$.
\end{enumerate}
\end{definition}

An independent set $I\in\mathcal{I}$ is called a \textit{basis} of the matroid $\mathcal{M}$ if $I$ is maximal. By the augmentation property, all bases of $\mathcal{M}$ share the same cardinality. This cardinality is defined as the \textit{rank} of $\mathcal{M}$, and the difference between the size of the set $X$ and the rank of $\mathcal{M}$ is defined as the \textit{corank} of $\mathcal{M}$.

It is not difficult to see that matrices give rise to a special class of matroids, the \textit{linear matroids} or the \textit{representable matroids}. Given a matrix $A$ over a field $\F$, let $X$ be the set of columns of $A$, and let $\mathcal{I}$ be a family of subsets of $X$ such that $I\in\mathcal{I}$ if and only if the column vectors in $I$ are linearly independent over $\F$. Then the rank of the linear matroid $\mathcal{M}=(X,\mathcal{I})$ coincides with the rank of the matrix $A$, and the corank of the matriod is the difference between the number of columns and the rank of $A$. Since the elements of a linear matroid are the columns of a matrix, a linear matroid is also called a \textit{column matroid}.

Given an ordered triple of positive integers $(n,r,b)$ such that $1\leq b\leq\binom{n}{r}$, our question about the existence of an $(n,r,b)$-matrix is in fact a stronger version of an open problem posed by Welsh \cite{welsh}: given such a triple $(n,r,b)$, does there always exist a matroid $\mathcal{M}=(X,\mathcal{I})$ such that $|X|=n$, the rank of $\mathcal{M}$ is $r$, and the number of bases in $\mathcal{I}$ is exactly $b$? Such a matroid is called an \textit{$(n,r,b)$-matroid}.

According to Mayhew and Royle \cite{mayhewroyle}, Anna de Mier answered Welsh's question negatively by proving that an $(n,r,b)$-matroid does not exist when $(n,r,b)=(6,3,11)$. This implies that a $(6,3,11)$-matrix does not exist. However, Mayhew and Royle conjectured that this is the lone counterexample.

\begin{conjecture}[\cite{mayhewroyle}]\label{conjmatroid}
Let $(n,r,b)\neq(6,3,11)$ be an ordered triple of positive integers such that $1\leq b\leq\binom{n}{r}$. Then an $(n,r,b)$-matroid exists.
\end{conjecture}

Here, we propose a stronger version of Conjecture~\ref{conjmatroid} as follows. Please note that this conjecture is merely being offered, but the complete proof is currently still out of reach.

\begin{conjecture}\label{conj}
Let $(n,r,b)\neq(6,3,11)$ be an ordered triple of positive integers such that $1\leq b\leq\binom{n}{r}$. Then an $(n,r,b)$-matrix exists.
\end{conjecture}

In this paper, we show that an $(n,r,b)$-matrix exists for a large family of triples $(n,r,b)$ by using an inductive argument, with base cases being taken care of by computer programming. Throughout this paper, $A$ denotes an $r\times n$ matrix over $\C$ with rank $r$ and corank $k=n-r$, and the number of $r\times r$ invertible submatrices of $A$ is $b$. Note that $b$ is an invariant if we perform elementary row operations on $A$ or permute the columns of $A$, so we can always assume that $A=(I_r|M)$, where $I_r$ is the identity matrix of order $r$ and $M$ is an $r\times k$ matrix. Furthermore, we only need to consider the existence of an $(n,r,b)$-matrix for $k\leq r$, due to the following simple observation.

\begin{proposition}\label{k<=r}
An $(n,r,b)$-matrix exists if and only if an $(n,k,b)$-matrix exists.
\end{proposition}

\begin{proof}
If an $(n,r,b)$-matrix exists, then an $(n,r,b)$-matroid exists. By duality, an $(n,k,b)$-matriod also exists. Since duality preserves representability, an $(n,k,b)$-matrix exists. The converse can be proved in a similar manner.
\end{proof}

In the rest of the paper, we are going to assume that $r$ is a positive integer such that $0\leq k\leq r$, or equivalently, $r\leq n\leq 2r$. We will also assume that $b$ is a positive integer such that $1\leq b\leq\binom{n}{r}=\binom{r+k}{k}$. To construct an $(n,r,b)$-matrix $A=(I_r|M)$, it is more convenient to consider the number of invertible square submatrices of $M$. A \emph{submatrix} of $M$ is a matrix formed by selecting a subset of the rows and a subset of the columns from $M$ and arranging those entries from both the selected rows and columns in the same relative positions. A \emph{square submatrix} of $M$ is formed if the subset of the rows and the subset of the columns selected share the same cardinality. The ``empty submatrix", formed by selecting an empty set of the rows and an empty set of the columns, is also considered as a square submatrix of $M$, and it is defined to be invertible.

It is worth noting that the number of square submatrices of $M$ is
$$\sum_{i=0}^k\binom{r}{i}\binom{k}{i}=\sum_{i=0}^k\binom{r}{i}\binom{k}{k-i}=\binom{r+k}{k},$$
which is precisely the number of $r\times r$ submatrices of $A=(I_r|M)$. In the following proposition, we are going to prove that there is a bijection between the set of invertible $r\times r$ submatrices of $A=(I_r|M)$ and the set of invertible square submatrices of $M$.

\begin{proposition}\label{AtoM}
Let $A=(I_r|M)$ be an $(n,r,b)$-matrix. Then the number of invertible square submatrices of $M$ is $b$.
\end{proposition}

\begin{proof}
Let $\mathcal{S}$ be the set of all invertible $r\times r$ submatrices of $A=(I_r|M)$, and let $\mathcal{T}$ be the set of all invertible square submatrices of $M$. Let $S\in\mathcal{S}$ be constructed by selecting columns $i_1,i_2,\dotsc,i_r$ from $A$, where $i_1,\dotsc,i_j\leq r$ and $i_{j+1},\dotsc,i_r>r$. Let $T$ be the square submatrix of $M$ with rows in $\{1,2,\dotsc,r\}\backslash\{i_1,i_2,\dotsc,i_j\}$ and columns $i_{j+1}-r,i_{j+2}-r,\dotsc,i_r-r$. It is easy to see that $|\det(S)|=|\det(T)|$, so $T\in\mathcal{T}$. It is also easy to see that our map from $\mathcal{S}$ to $\mathcal{T}$ by sending each $S$ to the corresponding $T$ has an inverse, which establishes our proposition.
\end{proof}

Let an $r\times k$ matrix $M$ with exactly $b$ invertible square submatrices be called an \textit{$(r,k,b)^*$-matrix}. In view of Propositions~\ref{k<=r} and \ref{AtoM}, Conjecture~\ref{conj} is equivalent to the following conjecture.

\begin{conjecture}\label{conjM}
Let $(r,k,b)\neq(3,3,11)$ be an ordered triple of nonnegative integers such that $\max\{1,k\}\leq r$ and $1\leq b\leq\binom{r+k}{k}$. Then an $(r,k,b)^*$-matrix exists.
\end{conjecture}

Next, we observe that the na\"{i}ve ``extension by zero" construction exhibits a useful relationship between the existence of various $(r,k,b)^*$-matrices.

\begin{lemma}\label{rkbigger}
If there exists an $(r_0,k_0,b)^*$-matrix, then for all integers $r$ and $k$ such that $k_0\leq k,r_0\leq r$, an $(r,k,b)^*$-matrix exists.
\end{lemma}

\begin{proof}
Let $M_0$ be an $(r_0,k_0,b)^*$-matrix. Let $M$ be an $r\times k$ matrix such that $M_0$ is a submatrix and all the extra entries are $0$'s. Then $M$ is an $(r,k,b)^*$-matrix.
\end{proof}

A similar statement is true for matroids as well. If there exists an $(n_0,r_0,b)$-matroid, then for all integers $r\geq r_0$ and $n\geq n_0+(r-r_0)$, an $(n,r,b)$-matroid exists by adding $r-r_0$ coloops and $(n-n_0)-(r-r_0)$ loops.

There are three main results in this paper. The first one is Theorem~\ref{blarge}, which states that Conjecture~$\ref{conjM}$ holds when $r$ and $b$ are large compared to $k$. With this tool, we are able to prove our second main result, namely Theorem~\ref{thmb<=r+3Cr}, which states that Conjecture~$\ref{conjM}$ holds under the additional condition that $b\leq\binom{r+3}{3}$. Finally, we strengthen these results into the following theorem, which is our third main result.

\begin{theorem}\label{rlarge}
For each fixed integer $k\geq3$, there exists $R\in\N$ such that for all integers $r\geq R$ and $1\leq b\leq\binom{r+k}{k}$, an $(r,k,b)^*$-matrix, and hence an $(r+k,r,b)$-matrix, exists.
\end{theorem}

To our knowledge, this paper gives the best answer to Welsh's problem thus far, and it provides a plausible reason for the existence of a counterexample for small values of $r$. Our techniques are mainly induction with algebraic constructions and analytic estimations, with the aid of computer programming only to verify the base cases.

\section{Existence of $(n,r,b)$-matrices with corank at most $2$}\label{corank2}

In this section, we will prove that $(r,k,b)^*$-matrices always exist if $k\leq2$. The cases $k=0$ and $k=1$ are very straightforward, and the case $k=2$ is the only nontrivial one. The following lemma in number theory will help us show the existence of $(r,2,b)^*$-matrices.

\begin{lemma}\label{sumofsq}
Let $s\geq5$ be a positive integer, and let $c\leq\frac{s^2-5s}{4}$ be a nonnegative integer. Then there exist nonnegative integers $a_1,a_2,\dotsc,a_s$ such that $a_1+a_2+\dotsb+a_s=s$ and $a_1^2+a_2^2+\dotsb+a_s^2=s+2c$.
\end{lemma}

\begin{proof}
For $5\leq s\leq32$, we verified the lemma with the following Mathematica code.\\
\\
\texttt{And @@ Table[\\
\indent SubsetQ[ Map[Total, IntegerPartitions[s, s]\^{}2],\\
\indent\indent Table[s + 2 c, \{c, Floor[(s\^{}2 - 5 s)/4]\}] ],\\
\indent \{s, 5, 32\}]}\\

For $s\geq33$, we proceed with strong induction on $s$. Suppose the statement is true for all integers $u$ such that $5\leq u<s$ for some $s\geq33$, i.e., for all nonnegative integers $c'\leq\frac{u^2-5u}{4}$, there exist nonnegative integers $a_1,a_2,\dotsc,a_u$ such that $a_1+a_2\dotsb+a_u=u$ and $a_1^2+a_2^2+\dotsb+a_u^2=u+2c'$.

Let $t$ and $c$ be integers such that $0<t\leq s-5$ and $0\leq c-\frac{t^2-t}{2}\leq\frac{(s-t)^2-5(s-t)}{4}$. Then $u:=s-t$ falls in the range $5\leq u<s$, and $c':=c-\frac{t^2-t}{2}\leq\frac{u^2-5u}{4}$. By the induction hypothesis, there exist nonnegative integers $a_1,a_2,\dotsc,a_{s-t}$ such that $a_1+a_2+\dotsb+a_{s-t}=s-t$ and $a_1^2+a_2^2+\dotsb+a_{s-t}^2=s -t+2\big(c-\frac{t^2-t}{2}\big)=s+2c-t^2$. If we set $a_{s-t+1}=t$ and $a_{s-t+2}=\dotsb=a_s=0$, then $a_1+\dotsb+a_s=s$ and $a_1^2+\dotsb+a_s^2=s+2c$, implying that the statement holds true for all integer $c$ satisfying $0\leq c-\frac{t^2-t}{2}\leq\frac{(s-t)^2-5(s-t)}{4}$, or equivalently, $\frac{t^2-t}{2}\leq c\leq\frac{3t^2-2st+3t+s^2-5s}{4}$.

It now suffices to show that the union of the intervals $I(t):=\Big[\frac{t^2-t}{2},\frac{3t^2-2st+3t+s^2-5s}{4}\Big]$ for $0<t\leq s-5$ covers $\Big[0,\frac{s^2-5s}{4}\Big]$ when $s\geq33$. Let $\alpha(t)=\frac{t^2-t}{2}$ and $\beta(t)=\frac{3t^2-2st+3t+s^2-5s}{4}$.\vspace{5pt}

\noindent\textit{Claim $1$.} $\frac{s^2-5s}{4}\leq\beta(t)$ if and only if $t\geq\frac{2}{3}s-1$, which is attainable for some $t$ in the range $0<t\leq s-5$ if $s\geq12$.

\begin{proof}[Proof of Claim $1$]
This inequality holds if and only if $3t^2-2st+3t\geq0$, which is equivalent to $t\geq\frac{2}{3}s-1$ since $t$ is positive. We finish by noticing that when $s\geq12$, $s-5\geq\frac{2}{3}s-1$.
\end{proof}

\noindent\textit{Claim $2$.} $\alpha(t-1)\leq\alpha(t)\leq\beta(t)$.

\begin{proof}[Proof of Claim $2$]
The first inequality holds since $\alpha(t)$ is an increasing function for $t\geq1$. The second inequality holds if and only if $(s-t)^2\geq5(s-t)$, which is always true since $0<t\leq s-5$.
\end{proof}

\noindent\textit{Claim $3$.} $\alpha(t)\leq\beta(t-1)$ if and only if $t\leq\frac{2s+1-\sqrt{16s+1}}{2}$.

\begin{proof}[Proof of Claim $3$]
This inequality holds if and only if $t^2-(2s+1)t+s^2-3s\geq0$, which occurs if and only if $t\leq\frac{2s+1-\sqrt{16s+1}}{2}$ or $t\geq\frac{2s+1+\sqrt{16s+1}}{2}$. However, $t\geq\frac{2s+1+\sqrt{16s+1}}{2}$ is rejected since $t<s$.
\end{proof}

By Claims $2$ and $3$, if $t\leq\frac{2s+1-\sqrt{16s+1}}{2}$, then $I(1)\cup\dotsb\cup I(t-1)\cup I(t)$ forms one closed interval. Hence, by Claim $1$, $\Big[0,\frac{s^2-5s}{4}\Big]\subseteq\overset{\left\lceil\frac{2}{3}s-1\right\rceil}{\underset{t=1}{\bigcup}}I(t)$ if and only if $\big\lceil\frac{2}{3}s-1\big\rceil\leq\frac{2s+1-\sqrt{16s+1}}{2}$. This inequality holds if $s$ is an integer satisfying $\frac{2}{3}s\leq\frac{2s+1-\sqrt{16s+1}}{2}$, or equivalently, $3\sqrt{16s+1}\leq2s+3$, which is true when $33s\leq s^2$, or $s\geq33$.
\end{proof}

\begin{theorem}\label{n<=r+2}
If $k\leq2$, then for all integers $b$ such that $1\leq b\leq\binom{r+k}{k}$, an $(r,k,b)^*$-matrix $M$ exists.
\end{theorem}
\begin{proof}
It is trivial for $k=0$. If $k=1$, then let $M$ be a column vector with the first $b-1$ entries $1$'s and the rest $0$'s.

If $k=2$, let the first $s$ entries in the first column of $M$ be all $1$'s, the first $s$ entries in the second column be nonzero, and the rest of the entries be all $0$'s. Furthermore, assume that there are $a_i$ $i$'s in the second column, $1\leq i\leq s$, where $a_1+a_2+\dotsb+a_s=s$. Then the number of invertible square submatrices of $M$ is
\begin{center}
\begin{tabular}{rl}
$1+2s+\sum_{i<j}a_ia_j$& \hspace{-8pt}$=1+2s+\frac{1}{2}(\sum_ia_i)^2-\frac{1}{2}(\sum_ia_i^2)$\\
&\hspace{-8pt}$=1+2s+\frac{1}{2}s^2-\frac{1}{2}(\sum_ia_i^2)$,
\end{tabular}
\end{center}
and we would like to set it to be $b$, which gives $s+2(\binom{s+2}{2}-b)=\sum_ia_i^2$.

By Lemma~\ref{sumofsq}, if $0\leq\binom{s+2}{2}-b\leq\frac{s^2-5s}{4}$, or equivalently $\frac{s^2+11s+4}{4}\leq b\leq\frac{s^2+3s+2}{2}$, there is a solution for $a_i$'s. Note that $\frac{(s+1)^2+11(s+1)+4}{4}\leq\frac{s^2+3s+2}{2}$ if $s\geq9$, so $\overset{r}{\underset{s=9}{\bigcup}}\left[\frac{s^2+11s+4}{4},\frac{s^2+3s+2}{2}\right]$ forms one closed interval $\left[46,\binom{r+2}{2}\right]$. From simple checking, the only integers $b$ satisfying $1\leq b\leq\binom{r+2}{2}$ that are not covered by $\overset{r}{\underset{s=5}{\bigcup}}\left[\frac{s^2+11s+4}{4},\frac{s^2+3s+2}{2}\right]$ lie in $[1,20]\cup[22,26]\cup[29,32]\cup[37,38]$. Hence, we can finish the proof by constructing $M$ explicitly for each of these values of $b$. In the following table, $\bz$ represents a column vector of all $0$'s (possibly of length $0$), which fills up the column so that $M$ has $r$ rows.
\begin{center}
\footnotesize\begin{tabular}{|c|c|c|c|c|c|c|c|c|}
\hline
$b=$& $1$& $2$& $3$& $4$& $5$& $6$& $7$& $8$\\
\hline
& & & & & & & &\\
$M=$& \begin{tabular}{|c|c|}
\hline
\bz& \bz\\
\hline
\end{tabular}& \begin{tabular}{|c|c|}
\hline
1&0\\
\bz& \bz\\
\hline
\end{tabular}& \begin{tabular}{|c|c|}
\hline
1&1\\
\bz& \bz\\
\hline
\end{tabular}& \begin{tabular}{|c|c|}
\hline
1&0\\
0&1\\
\bz& \bz\\
\hline
\end{tabular}& \begin{tabular}{|c|c|}
\hline
1&1\\
1&0\\
\bz& \bz\\
\hline
\end{tabular}& \begin{tabular}{|c|c|}
\hline
1&1\\
1&2\\
\bz& \bz\\
\hline
\end{tabular}& \begin{tabular}{|c|c|}
\hline
1&1\\
1&0\\
1&0\\
\bz& \bz\\
\hline
\end{tabular}& \begin{tabular}{|c|c|}
\hline
1&1\\
1&1\\
1&0\\
\bz& \bz\\
\hline
\end{tabular}\\
& & & & & & & &\\
\hline
\end{tabular}
\end{center}
\begin{center}
\footnotesize\begin{tabular}{|c|c|c|c|c|c|c|c|c|}
\hline
$b=$& $9$& $10$& $11$& $12$& $13$& $14$& $15$& $16$\\
\hline
& & & & & & & &\\
$M=$& \begin{tabular}{|c|c|}
\hline
1&1\\
1&2\\
1&0\\
\bz& \bz\\
\hline
\end{tabular}& \begin{tabular}{|c|c|}
\hline
1&1\\
1&2\\
1&3\\
\bz& \bz\\
\hline
\end{tabular}& \begin{tabular}{|c|c|}
\hline
1&1\\
1&1\\
1&0\\
1&0\\
\bz& \bz\\
\hline
\end{tabular}& \begin{tabular}{|c|c|}
\hline
1&1\\
1&2\\
1&0\\
1&0\\
\bz& \bz\\
\hline
\end{tabular}& \begin{tabular}{|c|c|}
\hline
1&1\\
1&1\\
1&2\\
1&0\\
\bz& \bz\\
\hline
\end{tabular}& \begin{tabular}{|c|c|}
\hline
1&1\\
1&2\\
1&3\\
1&0\\
\bz& \bz\\
\hline
\end{tabular}& \begin{tabular}{|c|c|}
\hline
1&1\\
1&2\\
1&3\\
1&4\\
\bz& \bz\\
\hline
\end{tabular}& \begin{tabular}{|c|c|}
\hline
1&1\\
1&1\\
1&0\\
1&0\\
0&1\\
\bz& \bz\\
\hline
\end{tabular}\\
& & & & & & & &\\
\hline
\end{tabular}
\end{center}
\begin{center}
\footnotesize\begin{tabular}{|c|c|c|c|c|c|c|c|c|}
\hline
$b=$& $17$& $18$& $19$& $20$& $22$& $23$& $24$& $25$\\
\hline
& & & & & & & &\\
$M=$& \begin{tabular}{|c|c|}
\hline
1&1\\
1&1\\
1&2\\
1&0\\
1&0\\
\bz& \bz\\
\hline
\end{tabular}& \begin{tabular}{|c|c|}
\hline
1&1\\
1&2\\
1&3\\
1&0\\
1&0\\
\bz& \bz\\
\hline
\end{tabular}& \begin{tabular}{|c|c|}
\hline
1&1\\
1&1\\
1&2\\
1&3\\
1&0\\
\bz& \bz\\
\hline
\end{tabular}& \begin{tabular}{|c|c|}
\hline
1&1\\
1&2\\
1&3\\
1&4\\
1&0\\
\bz& \bz\\
\hline
\end{tabular}& \begin{tabular}{|c|c|}
\hline
1&1\\
1&1\\
1&1\\
1&2\\
1&0\\
1&0\\
\bz& \bz\\
\hline
\end{tabular}& \begin{tabular}{|c|c|}
\hline
1&1\\
1&1\\
1&2\\
1&2\\
1&0\\
1&0\\
\bz& \bz\\
\hline
\end{tabular}& \begin{tabular}{|c|c|}
\hline
1&1\\
1&1\\
1&2\\
1&3\\
1&0\\
1&0\\
\bz& \bz\\
\hline
\end{tabular}& \begin{tabular}{|c|c|}
\hline
1&1\\
1&2\\
1&3\\
1&4\\
1&0\\
1&0\\
\bz& \bz\\
\hline
\end{tabular}\\
& & & & & & & &\\
\hline
\end{tabular}
\end{center}
\begin{center}
\footnotesize\begin{tabular}{|c|c|c|c|c|c|c|c|c|}
\hline
$b=$& $26$& $29$& $30$& $31$& $32$& $37$& $38$\\
\hline
& & & & & & &\\
$M=$& \begin{tabular}{|c|c|}
\hline
1&1\\
1&1\\
1&2\\
1&3\\
1&4\\
1&0\\
\bz& \bz\\
\hline
\end{tabular}& \begin{tabular}{|c|c|}
\hline
1&1\\
1&1\\
1&2\\
1&3\\
1&0\\
1&0\\
1&0\\
\bz& \bz\\
\hline
\end{tabular}& \begin{tabular}{|c|c|}
\hline
1&1\\
1&2\\
1&3\\
1&4\\
1&0\\
1&0\\
1&0\\
\bz& \bz\\
\hline
\end{tabular}& \begin{tabular}{|c|c|}
\hline
1&1\\
1&1\\
1&2\\
1&2\\
1&3\\
1&0\\
1&0\\
\bz& \bz\\
\hline
\end{tabular}& \begin{tabular}{|c|c|}
\hline
1&1\\
1&1\\
1&2\\
1&3\\
1&4\\
1&0\\
1&0\\
\bz& \bz\\
\hline
\end{tabular}& \begin{tabular}{|c|c|}
\hline
1&1\\
1&1\\
1&2\\
1&2\\
1&3\\
1&0\\
1&0\\
1&0\\
\bz& \bz\\
\hline
\end{tabular}& \begin{tabular}{|c|c|}
\hline
1&1\\
1&1\\
1&2\\
1&3\\
1&4\\
1&0\\
1&0\\
1&0\\
\bz& \bz\\
\hline
\end{tabular}\\
& & & & & & &\\
\hline
\end{tabular}
\end{center}
\end{proof}

\begin{corollary}\label{corb<=r+2Cr}
Conjecture~$\ref{conj}$ holds under the additional condition that $b\leq\binom{r+2}{2}$. In other words, for all ordered triples of positive integers $(n,r,b)$ such that $1\leq b\leq\min\left\{\binom{n}{r},\binom{r+2}{2}\right\}$, an $(n,r,b)$-matrix exists.
\end{corollary}

\begin{proof}
By Theorem~$\ref{n<=r+2}$ and Lemma~$\ref{rkbigger}$, for all $1\leq b\leq\binom{r+2}{2}$, and for all $r$ and $k$ such that $b\leq\binom{r+k}{k}$, an $(r,k,b)^*$-matrix exists. Consequently, an $(n,r,b)$-matrix exists for all $b$ such that $1\leq b\leq\min\left\{\binom{n}{r},\binom{r+2}{2}\right\}$.
\end{proof}

\section{Existence of $(n,r,b)$-matrices with large $b$}\label{manybases}

In Section~$\ref{corank2}$, we used induction together with computational exhaustion to show the existence of $(n,r,b)$-matrices with $1\leq b\leq\min\left\{\binom{n}{r},\binom{r+2}{2}\right\}$. The main tool was Lemma~\ref{sumofsq}, a number theoretic argument to count the number of invertible square submatrices of $M$. In this section, we count the number of singular square submatrices of $M$ instead. This is mostly done by choosing hyperplanes from a $k$-dimensional vector space and picking row vectors on the hyperplanes carefully.

\begin{definition}\label{regular}
Let $\ell$ be a nonnegative integer, and let $n_0,n_1,n_2,\dotsc,n_\ell$ be nonnegative integers such that $n_1,n_2,\dotsc,n_\ell\geq k$ and $n_0+n_1+\dotsb+n_\ell=r$. Let $\{1,2,\dotsc,r\}$ be partitioned into $\mathcal{I}_0,\mathcal{I}_1,\mathcal{I}_2,\dotsc,\mathcal{I}_\ell$ such that \begin{center}
$\mathcal{I}_i=\{1+\sum_{\iota=0}^{i-1}n_\iota,2+\sum_{\iota=0}^{i-1}n_\iota,\dotsc,\sum_{\iota=0}^in_\iota\}$
\end{center}
for each $0\leq i\leq\ell$. An $r\times k$ matrix $M$ is \textit{$(n_0,n_1,\dotsc,n_\ell)$-regular} if
\begin{enumerate}
\item every $j\times j$ submatrix is invertible for all $1\leq j<k$, and
\item a $k\times k$ submatrix is singular if and only if all rows of this submatrix come from the same $\mathcal{I}_i$ for some $1\leq i\leq\ell$.
\end{enumerate}
\end{definition}

\begin{proposition}\label{regularexist}
Given a nonnegative integer $\ell$ and nonnegative integers $n_0,n_1,n_2,\dotsc,n_\ell$ such that $n_1,n_2,\dotsc,n_\ell\geq k$ and $n_0+n_1+\dotsb+n_\ell=r$, an $(n_0,n_1,\dotsc,n_\ell)$-regular matrix $M$ exists.
\end{proposition}

\begin{proof}
Treat the set of $r\times k$ matrices with complex entries as the affine space $X=(\C^k)^{n_0}\times(\C^k)^{n_1}\times(\C^k)^{n_2}\times\dotsc\times(\C^k)^{n_\ell}$. For each $1\leq i\leq\ell$ and $1\leq j\leq k$, let $c_{ij}=i^{j-1}$. Since $(c_{ij})_{\ell\times k}$ is a Vandermonde matrix, any subset of at most $k$ vectors in $\{(c_{i1},c_{i2},\dotsc,c_{ik}):1\leq i\leq\ell\}$ is linearly independent. For each $1\leq i\leq\ell$, let $H_i$ be the hyperplane in $\C^k$ defined by $ c_{i1}x_1+c_{i2}x_2+\dotsb+c_{ik}x_k=0$.

Let $m_{\alpha\beta}, 1\leq\alpha\leq r, 1\leq\beta\leq k$, be the $(\alpha,\beta)$-coordinate functions of $X$. In other words, we treat $M$ as a matrix $(m_{\alpha\beta})_{r\times k}$ of functions, where every entry $m_{\alpha\beta}$ is a variable for all $1\leq\alpha\leq r$ and $1\leq\beta\leq k$. Then every determinant of square submatrices of $M$ is naturally a polynomial function on $X$. Consider the subspace $W=(\C^k)^{n_0} \times H_1^{n_1} \times \cdots \times H_\ell^{n_\ell}$ of $X$. Since $R=\C[m_{\alpha\beta}:1\leq\alpha\leq r,1\leq\beta\leq k]$ is the ring of regular functions on $X$, we can naturally define $Q=\C[m_{\alpha\beta}:1\leq\alpha\leq r,1\leq\beta\leq k]/\langle c_{i1}m_{\alpha1}+c_{i2}m_{\alpha2}+\dotsb+c_{ik}m_{\alpha k}=0\text{ for all }\alpha\in\mathcal{I}_i,\text{ where }1\leq i\leq\ell\rangle$ as the ring of regular functions on $W$. For any $f\in R$, denote its restriction on $W$ by $f_W \in Q$.


For any $j\times j$ submatrix $B$ of $M$, let $B$ come from rows $\alpha_1,\alpha_2,\dotsc,\alpha_j$ and columns $\beta_1,\beta_2,\dotsc,\beta_j$ in $M$. Assume that $\alpha_t\in\mathcal{I}_{i_t}$ for each $1\leq t\leq j$, where $i_1\leq i_2\leq\dotsb\leq i_j$. Furthermore, assume that $\beta_1<\beta_2<\dotsb<\beta_j$. We would like to show that $(\det B)_W$ is a nonzero function in $Q$ if either $i_1<i_j$ or $i_1=i_j$ and $j<k$, so that the complement of the zero locus of $(\det B)_W$ defines a nonempty Zariski open subset of $W$.

If $i_1<i_j$, consider a $j\times j$ matrix
$$\widetilde{B}=\begin{pmatrix}
-i_1^{\beta_j-\beta_1}&0&\dotsb&0&\dotsb&0&1\\
-i_2^{\beta_{j-1}-\beta_1}&0&\dotsb&0&\dotsb&1&0\\
\vdots&\vdots&\udots&\vdots&\udots&\vdots&\vdots\\
-i_t^{\beta_{j-t+1}-\beta_1}&0&\dotsb&1&\dotsb&0&0\\
\vdots&\vdots&\udots&\vdots&\udots&\vdots&\vdots\\
-i_{j-1}^{\beta_2-\beta_1}&1&\dotsb&0&\dotsb&0&0\\
i_j^{\beta_j-\beta_1}+i_j^{\beta_{j-1}-\beta_1}+\dotsb+i_j^{\beta_2-\beta_1}&-1&\dotsb&-1&\dotsb&-1&-1\\
\end{pmatrix}.$$
When calculating the determinant of $\widetilde{B}$, we can add the first $j-1$ rows to the last, so $|\det\widetilde{B}|=\sum_{t=1}^{j-1}i_j^{\beta_{j-t+1}-\beta_1}-\sum_{t=1}^{j-1}i_t^{\beta_{j-t+1}-\beta_1}>0$. Let $\widetilde{M}$ be the $r \times k$ matrix such that its $j \times j$ submatrix with rows $\alpha_1,\alpha_2,\dotsc,\alpha_j$ and columns $\beta_1,\beta_2,\dotsc,\beta_j$ is $\widetilde{B}$, and all its other entries are zero. Note that $\widetilde{M}$ lies on $W$, since the $\alpha_t$-th row vector is orthogonal to the vector $(c_{i_t1},c_{i_t2},\dotsc,c_{i_tk})$. Also, when $\det B$ is evaluated at $\widetilde{M}$, the result is a nonzero value $\det \widetilde{B}$. Thus, $(\det B)_W$ is a nonzero function in $Q$. 

If $i_1=i_j$ and $j<k$, then let $\beta_0\leq k$ be a positive integer distinct from $\beta_1,\beta_2,\dotsc,\beta_j$. Let $\doublewidetilde{M}$ be the $r\times k$ matrix such that its $j\times j$ submatrix with rows $\alpha_1,\alpha_2,\dotsc,\alpha_j$ and columns $\beta_1,\beta_2,\dotsc,\beta_j$ is the $j \times j$ identity matrix, its $(\alpha_t,\beta_0)$-entry takes the value $-i_1^{\beta_t-\beta_0}$ for all $1\leq t\leq j$, and all its other entries are zero. Then we have a similar conclusion that $\doublewidetilde{M}$ lies on $W$, and that when $\det B$ is evaluated at $\doublewidetilde{M}$, the result is $1$. Thus, $(\det B)_W$ is again a nonzero function in $Q$.


In conclusion, the complement of the zero locus of each such $(\det B)_W$ defines a nonempty Zariski open subset of $W$. It is well-known that all nonempty Zariski open subsets of $W$ are dense in $W$, so any finite intersection of such is still nonempty. Therefore, the common intersection $U$ of the complements of the zero loci of $(\det B)_W$, where $B$ runs through all $j \times j$ submatrices of $M$ for $1\leq j<k$ and all $k \times k$ submatrices of $M$ whose rows do not come from the same $\mathcal{I}_i$ for some $1\leq i\leq\ell$, is nonempty. 

Finally, from our construction, every matrix from $W$ satisfies the condition that all $k \times k$ submatrices whose rows come from the same $\mathcal{I}_i$ are singular, since the $k$ row vectors lie on the same hyperplane $H_i$ and must be linearly dependent. As a result, any point on $U$ forms an $(n_0,n_1,\dotsc,n_\ell)$-regular matrix.

\end{proof}

Given an ordered triple $(r,k,b)$, if we want to construct an $(r,k,b)^*$-matrix, we first try to look for nonnegative integers $a_k,a_{k+1},\dotsc,a_r$ such that $\overset{r}{\underset{s=k}{\sum}}a_ss\leq r$ and $\overset{r}{\underset{s=k}{\sum}}a_s\binom{s}{k}=\binom{r+k}{k}-b=\overline{b}$, which denotes the number of singular square submatrices in $M$. If such nonnegative integers exist, then we define $(n_0,n_1,\dotsc,n_\ell)$ as follows: let $n_0=r-\overset{r}{\underset{s=k}{\sum}}a_ss\geq0$, and for each $s$ in the range $k\leq s\leq r$, let $n_i=s$ for all $i$ that satisfies $1+\overset{s-1}{\underset{j=k}{\sum}}a_j\leq i\leq\overset{s}{\underset{j=k}{\sum}}a_j$. By Proposition~\ref{regularexist}, an $(n_0,n_1,\dotsc,n_\ell)$-regular matrix $M$ exists, and this $M$ is an $(r,k,b)^*$-matrix by our construction. In particular, since the only singular square submatrices in $M$ are those $k\times k$ submatrices with all rows coming from the same $\mathcal{I}_i$, we have $\overline{b}=\overset{\ell}{\underset{i=1}{\sum}}\binom{n_i}{k}=\overset{r}{\underset{s=k}{\sum}}a_s\binom{s}{k}$.

\begin{proposition}\label{AsymEst}
For each fixed integer $k\geq3$, there exists $r_0\in\N$ such that for all integers $r\geq r_0$ and $0\leq\overline{b}\leq\binom{r+k-1}{k-1}$, there exist nonnegative integers $a_k,a_{k+1},\dotsc,a_r$ such that

\textup{(}i\textup{)} $\overset{r}{\underset{s=k}{\sum}}a_s\binom{s}{k}=\overline{b}$, and

\textup{(}ii\textup{)} $\overset{r}{\underset{s=k}{\sum}}a_ss\leq r$.
\end{proposition}

\begin{proof}
Let $K=(k^{\frac{1}{k}}+1)^k+1$. Note that $K>(k^{\frac{1}{k}}+1)^k>\big(k^{\frac{1}{k}}\big)^k+k\cdot k^{\frac{1}{k}}>2k$. Let $\widetilde{m}=\big\lceil\big(\log\log(r+k-1)-\log\log\big(\tfrac{K}{(k^{1/k}+1)^k}\big)\big)\big/\log\big(\tfrac{k}{k-1}\big)-1\big\rceil$. Since
$$\widetilde{m}\big(k^{\frac{1}{k}}+1\big)(r+k-1)^{\frac{k-1}{k}}+k\tbinom{K}{k}=O\big(r^{\frac{k-1}{k}}\log\log r\big)=o(r),$$
there exists $r_0\in\N$ such that for all integers $r\geq r_0$, $r\geq\widetilde{m}\big(k^{\frac{1}{k}}+1\big)(r+k-1)^{\frac{k-1}{k}}+k\tbinom{K}{k}.$ From this point onwards in this proof, consider $r\geq r_0$.

Let $\overline{b}$ be an integer such that $0\leq\overline{b}\leq\binom{r+k-1}{k-1}$. Let $\overline{b}_1=\overline{b}$. For each integer $i\geq1$, if $\overline {b}_i\geq\binom{K}{k}$, then let $s_i\in\N$ be such that $\binom{s_i}{k}\leq\overline{b}_i<\binom{s_i+1}{k}$. Note that $1\leq\overline{b}_i\big/\binom{s_i}{k}<\binom{s_i+1}{k}\big/\binom{s_i}{k}$, which is less than $2$ since $\binom{2k}{k}<\binom{K}{k}<\binom{s_i+1}{k}$, implying that $2k<s_i+1$, and hence $\binom{s_i+1}{k}=\frac{s_i+1}{s_i-k+1}\binom{s_i}{k}<\frac{2k}{k}\binom{s_i}{k}=2\binom{s_i}{k}$. Let $a_{s_i}=1$ and $\overline{b}_{i+1}=\overline{b}_i-\binom{s_i}{k}$. If $\overline{b}_i<\binom{K}{k}$ for some $i\geq1$, then let $a_k=\overline{b}_i$ and all other undetermined $a_s$ be $0$.

From this definition, condition (\textit{i}) is clearly satisfied. For condition (\textit{ii}), $\overset{r}{\underset{s=k}{\sum}}a_ss=s_1+s_2+\dotsb+s_m+\overline{b}_{m+1}\cdot k$ for some $m\geq0$. To give an upper bound to this sum, let $f_k(x)=\binom{x}{k}$ and $f_{k-1}(x)=\binom{x}{k-1}$ be functions on the real line. Note that they are both strictly increasing functions when $x\geq k$, so the function $f_k$ has an inverse $f_k^{-1}$ when $x\geq k$, and the composition $f_k^{-1}\circ f_{k-1}$ is also a strictly increasing function.

For all $i\geq1$, $\overline{b}_{i+1}=\overline{b}_i-\binom{s_i}{k}<\binom{s_i+1}{k}-\binom{s_i}{k}=\binom{s_i}{k-1}$. Hence,
\begin{equation}\label{eq:sirepeat}
s_1\leq(f_k^{-1}\circ f_{k-1})(r+k-1)\text{ and }s_{i+1}\leq f_k^{-1}(\overline{b}_{i+1})<(f_k^{-1}\circ f_{k-1})(s_i).
\end{equation}
When $x>2k$,
$$f_k^{-1}\circ f_{k-1}(x)<f_k^{-1}\Big(\tfrac{x^{k-1}}{(k-1)!}\Big)=f_k^{-1}\Big(\tfrac{\left(k^{1/k}x^{(k-1)/k}\right)^k}{k!}\Big)<k^{\frac{1}{k}}x^{\frac{k-1}{k}}+k-1.$$
Notice that $k-1<2^{k-1}$, so $(k-1)^k<2^{k-1}k^{k-1}$, implying $k-1<(2k)^{\frac{k-1}{k}}<x^{\frac{k-1}{k}}$. Therefore,
\begin{equation}\label{eq:fkineq}
f_k^{-1}\circ f_{k-1}(x)<\big(k^{\frac{1}{k}}+1\big)x^{\frac{k-1}{k}}.
\end{equation}
By the definition of $r$ and $s_i$, we have $r+k-1\geq K>2k$ and $s_i\geq K>2k$ for all $1\leq i\leq m$. Thus, combining inequalities $\eqref{eq:sirepeat}$ and $\eqref{eq:fkineq}$, we have
$$s_1<\big(k^{\frac{1}{k}}+1\big)(r+k-1)^{\frac{k-1}{k}}$$
and
\begin{align}\label{eq:siineq}
s_i\leq(f_k^{-1}\circ f_{k-1})^i(r+k-1)& <\big(k^{\frac{1}{k}}+1\big)^{1+\frac{k-1}{k}+(\frac{k-1}{k})^2+\dotsb+(\frac{k-1}{k})^{i-1}}(r+k-1)^{(\frac{k-1}{k})^i}\nonumber\\
& <\big(k^{\frac{1}{k}}+1\big)^k(r+k-1)^{(\frac{k-1}{k})^i}.
\end{align}

Assume that $m>\widetilde{m}$. By inequality \eqref{eq:siineq}, $(f_k^{-1}\circ f_{k-1})(s_{\widetilde{m}})<\big(k^{\frac{1}{k}}+1\big)^k(r+k-1)^{(\frac{k-1}{k})^{\widetilde{m}+1}}$, which is less than or equal to $K$ by the definition of $\widetilde{m}$. This implies that $\binom{s_{\widetilde{m}}}{k-1}<\binom{K}{k}$. Since $\overline{b}_{\widetilde{m}+1}<\binom{s_{\widetilde{m}}}{k-1}$, we have $\overline{b}_{\widetilde{m}+1}<\binom{K}{k}$. However, $m$ is by definition the least nonnegative integer such that $\overline{b}_{m+1}<\binom{K}{k}$, which is a contradiction. Therefore, $m\leq\widetilde{m}$, and
\begin{center}
$\overset{r}{\underset{s=k}{\sum}}a_ss=s_1+s_2+\dotsb+s_m+\overline{b}_{m+1}\cdot k\leq ms_1+k\binom{K}{k}<\widetilde{m}\big(k^{\frac{1}{k}}+1\big)(r+k-1)^{\frac{k-1}{k}}+k\binom{K}{k}$,
\end{center}
which does not exceed $r$ when $r\geq r_0$. In other words, condition (\textit{ii}) is also satisfied.
\end{proof}

This immediately gives the following asymptotic result on the existence of $(n,r,b)$-matrices.

\begin{theorem}\label{blarge}
For each fixed integer $k\geq3$, let $r_0\in\N$ be as given by Proposition~$\ref{AsymEst}$. Then for all integers $r\geq r_0$ and $\binom{r_0+k-1}{k}\leq b\leq\binom{n}{r}$, an $(n,r,b)$-matrix exists, where $n=r+k$.
\end{theorem}

\begin{proof}
By Proposition~$\ref{AsymEst}$, for every integer $r\geq r_0$, for all integers $b$ such that $\binom{r+k-1}{k}=\binom{r+k}{k}-\binom{r+k-1}{k-1}\leq b\leq\binom{r+k}{k}$, we can construct an $(r,k,b)^*$-matrix following the procedures introduced before Proposition~\ref{AsymEst}. Finally, we are done by noticing that the intervals $\big[\binom{r_0+k-1+i}{k},\binom{r_0+k+i}{k}\big]$, $i=0,1,\dotsc,r-r_0$, cover all integers $\binom{r_0+k-1}{k}\leq b\leq\binom{r+k}{k}=\binom{n}{r}$.
\end{proof}

\section{Existence of $(n,r,b)$-matrices with corank at most $3$}\label{corank3}

In this section, we will prove that apart from $(r,k,b)=(3,3,11)$, $(r,k,b)^*$-matrices always exist if $k=3$. To do so, we first refine Proposition~$\ref{AsymEst}$ for the case $k=3$ by finding an explicit value for $r_0$.

\begin{proposition}\label{Indk=3}
For all integers $r\geq49$ and $0\leq\overline{b}\leq\binom{r+2}{2}$, there exist nonnegative integers $a_3,a_4,\dotsc,a_r$ such that

\textup{(}i\textup{)} $\overset{r}{\underset{s=3}{\sum}}a_s\binom{s}{3}=\overline{b}$, and

\textup{(}ii\textup{)} $\overset{r}{\underset{s=3}{\sum}}a_ss\leq r$.
\end{proposition}

\begin{proof}
Let the statement of the proposition be denoted by $\mathcal{P}(r)$ for $r\geq49$. We checked computationally with Mathematica that $\mathcal{P}(r)$ holds for $49\leq r\leq203$ (see Appendix~\ref{49<=r<=203} for the program code). Suppose that $\mathcal{P}(r')$ is true for all $49\leq r'\leq r$ for some $r\geq203$. We now check the validity of $\mathcal{P}(r +1)$.

Let $\overline{b}$ be an integer such that $0\leq\overline{b}\leq\binom{r+3}{2}$. If $\overline{b}\leq\binom{r+2}{2}$, we can apply $\mathcal{P}(r)$ to $\overline{b}$ to obtain nonnegative integers $a_3,a_4,\dotsc,a_{r}$ and $a_{r+1}=0$ such that

(\textit{i}) $\overset{r+1}{\underset{s=3}{\sum}}a_s\binom{s}{3}=\overline{b}$, and

(\textit{ii}) $\overset{r+1}{\underset{s=3}{\sum}}a_ss\leq r\leq r+1$,

\noindent so it suffices to consider $\binom{r+2}{2}<\overline{b}\leq\binom{r+3}{2}$. Note that there exists a unique integer $s_0$ such that $\binom{s_0}{3}\leq\overline{b}<\binom{s_0+1}{3}$.

Set $\overline{b}':=\overline{b}-\binom{s_0}{3}$. If $\overline{b}'=0$, we are done. If $\overline{b}'\geq1$, we have $1\leq\overline{b}'<\binom{s_0+1}{3}-\binom{s_0}{3}=\binom{s_0}{2}$. Clearly, $s_0\leq r+2$, so we can apply the induction hypothesis $\mathcal{P}(s_0-2)$ to $\overline{b}'$ as long as $s_0\geq51$. This allows us to obtain nonnegative integers $a_3,a_4,\dotsc,a_{s_0-2}$ such that

(\textit{i}) $\overset{s_0-2}{\underset{s=3}{\sum}}a_s\binom{s}{3}=\overline{b}'$, and

(\textit{ii}) $\overset{s_0-2}{\underset{s=3}{\sum}}a_ss\leq s_0-2$.

\noindent By setting $a_{s_0-1}=0$, $a_{s_0}=1$, and $a_{s_0+1}=\dotsb=a_{r +1}=0$, we have

(\textit{i}) $\overset{r+1}{\underset{s=3}{\sum}}a_s\binom{s}{3}=\overline{b}$, and

(\textit{ii}) $\overset{r+1}{\underset{s=3}{\sum}}a_ss\leq2s_0-2$.

It suffices to show that $s_0\geq51$ and $2s_0-2\leq r+1$. First, observe that when $r\geq203$, we have
$$\tbinom{s_0+1}{3}>\overline{b}>\tbinom{r+2}{2}\geq\tbinom{205}{2}=20910,$$
or equivalently, $(s_0+1)s_0(s_0-1)>125460$. A straightforward calculation shows that this inequality holds if and only if $s_0\geq51$.

Next, for any integer $x\geq1$,
\begin{align*}
\tbinom{\frac{x+3}{2}}{3}-\tbinom{x+3}{2}& =\tfrac{(x+3)(x+1)(x-1)}{48}-\tfrac{(x+3)(x+2)}{2}\\
&=\tfrac{1}{48}(x+3)(x^2-24x-49)\geq0
\end{align*}
if and only if $x^2-24x-49\geq0$. By solving the quadratic inequality, it is easy to see that it holds for all integers $x\geq26$. Since $r\geq203$, we have
\begin{center}
$\binom{s_0}{3}\leq\overline{b}\leq\binom{r+3}{2}\leq\binom{\frac{r+3}{2}}{3}.$
\end{center}
This implies $\frac{r+3}{2}\geq s_0$, or equivalently, $2s_0-2\leq r+1$, since $\binom{x}{3}$ is a strictly increasing function for $x\geq3$.
\end{proof}

To complete the case for $k=3$, we still need to consider $3\leq r\leq48$. We have to slightly modify the construction of $M$ based on the one introduced before Proposition~\ref{AsymEst}.

Let $M$ be partitioned into $\ell+1$ submatrices $M_0,M_1,\dotsc,M_\ell$, where $M_i$ denotes the submatrix of $M$ with all rows in $\mathcal{I}_i$. Recall from the proof of Proposition~\ref{regularexist} that all the rows in $M_i$, $1\leq i\leq\ell$, form a plane $H_i$ of dimension $2$. Let the normal vector of $H_i$ be $(c_{i1},c_{i2},c_{i3})$.

Here are three types of modifications on $M_i$.
\begin{enumerate}
\item\label{type1} For some $1\leq i\leq\ell$, take $c_{i3}=0$, but $c_{i1}$ and $c_{i2}$ are nonzero. Let the row vectors in $M_i$ be $(m_{i1},m_{i2},m_{\alpha3})$ for all $\alpha\in\mathcal{I}_i$ such that $c_{i1}m_{i1}+c_{i2}m_{i2}=0$, $m_{\alpha3}$ are all distinct, and all entries in $M_i$ are nonzero. In this $M_i$, every $2\times2$ submatrix obtained from the first two columns is singular. Hence, the number of singular submatrices in $M_i$ increases from $\binom{n_i}{3}$ to $\binom{n_i}{3}+\binom{n_i}{2}=\binom{n_i+1}{3}$. In other words, we save one row every time we use such a modified $M_i$.
\item\label{type2} For at most three different values of $i$ between $1$ and $\ell$ inclusively, take exactly two of $c_{i1},c_{i2},c_{i3}$ to be zero. For example, we can take $c_{i1}=c_{i2}=0$. Let the row vectors in $M_i$ be $(m_{\alpha1},m_{\alpha2},0)$ for all $\alpha\in\mathcal{I}_i$ such that they are pairwisely linearly independent, and all $m_{\alpha1}$ and $m_{\alpha2}$ are nonzero. In this $M_i$, there are $n_i$ singular $1\times1$ submatrices and $2\binom{n_i}{2}$ singular $2\times2$ submatrices. Hence, the number of singular submatrices in $M_i$ increases from $\binom{n_i}{3}$ to $\binom{n_i}{3}+2\binom{n_i}{2}+n_i=\binom{n_i+2}{3}$. In other words, we save two rows for up to three times if we use such a modified $M_i$.
\item\label{type3} In $M_0$, replace $t$ distinct rows by $t$ identical copies of the same row. Then every $3\times3$ and $2\times2$ submatrix among these identical vectors is singular, and every $3\times3$ submatrix formed by picking two rows from these $t$ identical vectors and one row outside these vectors is also singular. Hence, $\overline{b}$ increases by $\binom{t}{3}+\binom{t}{2}(r-t+3)$.
\end{enumerate}

We present the following example to illustrate all three modifications of $M_i$.

\begin{example}\label{example}
Consider $(r,k,b)=(15,3,699)$. Then $\overline{b}=\binom{15+3}{3}-699=117$. Since there do not exist nonnegative integers $a_3,a_4,\dotsc,a_{15}$ such that $\underset{s=3}{\overset{15}{\sum}}a_s\binom{s}{3}=117$ and $\underset{s=3}{\overset{15}{\sum}}a_ss\leq15$, we cannot produce $M$ solely based on the  constructions introduced before Proposition~\ref{AsymEst}. For instance, the greedy algorithm gives $\binom{9}{3}+\binom{6}{3}+\binom{5}{3}+3\binom{3}{3}=117$, but $9+6+5+3\cdot3=29>15$. Here, we present a construction of $M$ with all three modifications used.

Note that $\binom{4}{3}+\binom{4}{2}(15-4+3)=88$ and $\binom{2}{3}+\binom{2}{2}(15-2+3)=16$, so we are going to use type~$\ref{type3}$ modification twice: once with $t=4$ and once with $t=2$. Hence, we can take the two type~$\ref{type3}$ modifications as
\begin{center}
{\footnotesize\begin{tabular}{|c|c|c|}
\hline
$1$& $2$& $4$\\
$1$& $2$& $4$\\
$1$& $2$& $4$\\
$1$& $2$& $4$\\
\hline
\end{tabular}} and {\footnotesize\begin{tabular}{|c|c|c|}
\hline
$1$& $3$& $9$\\
$1$& $3$& $9$\\
\hline
\end{tabular}}.
\end{center}
Since $117-88-16=13=\binom{3+2}{3}+\binom{1+2}{3}+\binom{1+2}{3}+\binom{2+1}{3}$, we are going to use type~$\ref{type2}$ modification three times and type~$\ref{type1}$ modification once to finish the construction. We can take the three type~$\ref{type2}$ modifications and type~$\ref{type1}$ modification as
\begin{center}
{\footnotesize\begin{tabular}{|c|c|c|}
\hline
$2$& $1$& $0$\\
$3$& $1$& $0$\\
$4$& $1$& $0$\\
\hline
\end{tabular}}, {\footnotesize\begin{tabular}{|c|c|c|}
\hline
$5$& $0$& $1$\\
\hline
\end{tabular}}, {\footnotesize\begin{tabular}{|c|c|c|}
\hline
$0$& $6$& $1$\\
\hline
\end{tabular}}, and {\footnotesize\begin{tabular}{|c|c|c|}
\hline
$2$& $3$& $7$\\
$2$& $3$& $10$\\
\hline
\end{tabular}}
\end{center}
respectively. Finally, note that we have only used $4+2+3+1+1+2=13$ rows, we need to take two more rows in $M_0$, say
\begin{center}
{\footnotesize\begin{tabular}{|c|c|c|}
\hline
$1$& $1$& $1$\\
\hline
\end{tabular}} and {\footnotesize\begin{tabular}{|c|c|c|}
\hline
$8$& $7$& $1$\\
\hline
\end{tabular}}.
\end{center}
Altogether, our matrix $M$ can be given by
$$M^\top={\footnotesize\begin{tabular}{|ccccccccccccccc|}
\hline
$1$& $8$& $1$& $1$& $1$& $1$& $1$& $1$& $2$& $3$& $4$& $5$& $0$& $2$& $2$\\
\hline
$1$& $7$& $2$& $2$& $2$& $2$& $3$& $3$& $1$& $1$& $1$& $0$& $6$& $3$& $3$\\
\hline
$1$& $1$& $4$& $4$& $4$& $4$& $9$& $9$& $0$& $0$& $0$& $1$& $1$& $7$& $10$\\
\hline
\end{tabular}}.$$
\end{example}

The construction process introduced in Example~\ref{example} allows us to construct $(r,3,b)^*$ matrices $M$ for more general $(r,3,b)$.

\begin{proposition}\label{Basek=3}
For all integers $11\leq r\leq48$ and $0\leq\overline{b}\leq\binom{r+2}{2}$, an $\big(r,3,\binom{r+3}{3}-\overline{b}\big)^*$-matrix $M$ exists.
\end{proposition}

\begin{proof}
Let $r$ and $\overline{b}$ be integers such that $11\leq r\leq48$ and $0\leq\overline{b}\leq\binom{r+2}{2}$. Let $\overline{b}_0=\overline{b}$. For each integer $i\geq0$, if $\overline{b}_i\geq\binom{2}{3}+\binom{2}{2}(r-2+3)=r+1$, then let $t_i$ be the largest integer such that $\binom{t_i}{3}+\binom{t_i}{2}(r-t_i+3)\leq\overline{b}_i$, and let $\overline{b}_{i+1}=\overline{b}_i-\big(\binom{t_i}{3}+\binom{t_i}{2}(r-t_i+3)\big)$. We repeat this process until $\overline{b}_j<\binom{2}{3}+\binom{2}{2}(r-2+3)=r+1$ for some $j\geq0$. In this process, we are building a $(t_0+t_1+\dotsb+t_{j-1})\times3$ matrix using only rows from a type~$\ref{type3}$ modification.

After using only rows from a type~$\ref{type3}$ modification, we still need to pick appropriate rows to give an additional $\overline{b}_j$ singular submatrices if $\overline{b}_j>0$. Note that $\overline{b}_j\leq r\leq48$. The following is a list of partitions of integers from $1$ to $48$. Each summand $\binom{s+2}{3}$ corresponds to $s$ rows from a type~$\ref{type2}$ modification, and each summand $\binom{s+1}{3}$ corresponds to $s$ rows from a type~$\ref{type1}$ modification. The bold digit in each of the following binomial coefficients gives the value of $s$.\\

\noindent\begin{tabular}{p{215pt}p{200pt}}
\hspace{-5pt}$1=\binom{\textbf{1}+2}{3}$& $25=\binom{\textbf{4}+2}{3}+\binom{\textbf{2}+2}{3}+\binom{\textbf{1}+2}{3}$\\
\end{tabular}\\
\noindent\begin{tabular}{p{215pt}p{200pt}}
\hspace{-5pt}$2=\binom{\textbf{1}+2}{3}+\binom{\textbf{1}+2}{3}$& $26=\binom{\textbf{4}+2}{3}+\binom{\textbf{2}+2}{3}+\binom{\textbf{1}+2}{3}+\binom{\textbf{2}+1}{3}$\\
\end{tabular}\\
\noindent\begin{tabular}{p{215pt}p{200pt}}
\hspace{-5pt}$3=\binom{\textbf{1}+2}{3}+\binom{\textbf{1}+2}{3}+\binom{\textbf{1}+2}{3}$& $27=\binom{\textbf{4}+2}{3}+\binom{\textbf{2}+2}{3}+\binom{\textbf{1}+2}{3}+\binom{\textbf{2}+1}{3}+\binom{\textbf{2}+1}{3}$\\
\end{tabular}\\
\noindent\begin{tabular}{p{215pt}p{200pt}}
\hspace{-5pt}$4=\binom{\textbf{2}+2}{3}$& $28=\binom{\textbf{4}+2}{3}+\binom{\textbf{2}+2}{3}+\binom{\textbf{2}+2}{3}$\\
\end{tabular}\\
\noindent\begin{tabular}{p{215pt}p{200pt}}
\hspace{-5pt}$5=\binom{\textbf{2}+2}{3}+\binom{\textbf{1}+2}{3}$& $29=\binom{\textbf{4}+2}{3}+\binom{\textbf{2}+2}{3}+\binom{\textbf{2}+2}{3}+\binom{\textbf{2}+1}{3}$\\
\end{tabular}\\
\noindent\begin{tabular}{p{215pt}p{200pt}}
\hspace{-5pt}$6=\binom{\textbf{2}+2}{3}+\binom{\textbf{1}+2}{3}+\binom{\textbf{1}+2}{3}$& $30=\binom{\textbf{4}+2}{3}+\binom{\textbf{3}+2}{3}$\\
\end{tabular}\\
\noindent\begin{tabular}{p{215pt}p{200pt}}
\hspace{-5pt}$7=\binom{\textbf{2}+2}{3}+\binom{\textbf{1}+2}{3}+\binom{\textbf{1}+2}{3}+\binom{\textbf{2}+1}{3}$& $31=\binom{\textbf{4}+2}{3}+\binom{\textbf{3}+2}{3}+\binom{\textbf{1}+2}{3}$\\
\end{tabular}\\
\noindent\begin{tabular}{p{215pt}p{200pt}}
\hspace{-5pt}$8=\binom{\textbf{2}+2}{3}+\binom{\textbf{2}+2}{3}$& $32=\binom{\textbf{4}+2}{3}+\binom{\textbf{3}+2}{3}+\binom{\textbf{1}+2}{3}+\binom{\textbf{2}+1}{3}$\\
\end{tabular}\\
\noindent\begin{tabular}{p{215pt}p{200pt}}
\hspace{-5pt}$9=\binom{\textbf{2}+2}{3}+\binom{\textbf{2}+2}{3}+\binom{\textbf{1}+2}{3}$& $33=\binom{\textbf{4}+2}{3}+\binom{\textbf{3}+2}{3}+\binom{\textbf{1}+2}{3}+\binom{\textbf{2}+1}{3}+\binom{\textbf{2}+1}{3}$\\
\end{tabular}\\
\noindent\begin{tabular}{p{215pt}p{200pt}}
\hspace{-5pt}$10=\binom{\textbf{3}+2}{3}$& $34=\binom{\textbf{4}+2}{3}+\binom{\textbf{3}+2}{3}+\binom{\textbf{2}+2}{3}$\\
\end{tabular}\\
\noindent\begin{tabular}{p{215pt}p{200pt}}
\hspace{-5pt}$11=\binom{\textbf{3}+2}{3}+\binom{\textbf{1}+2}{3}$& $35=\binom{\textbf{5}+2}{3}$\\
\end{tabular}\\
\noindent\begin{tabular}{p{215pt}p{200pt}}
\hspace{-5pt}$12=\binom{\textbf{3}+2}{3}+\binom{\textbf{1}+2}{3}+\binom{\textbf{1}+2}{3}$& $36=\binom{\textbf{5}+2}{3}+\binom{\textbf{1}+2}{3}$\\
\end{tabular}\\
\noindent\begin{tabular}{p{215pt}p{200pt}}
\hspace{-5pt}$13=\binom{\textbf{3}+2}{3}+\binom{\textbf{1}+2}{3}+\binom{\textbf{1}+2}{3}+\binom{\textbf{2}+1}{3}$& $37=\binom{\textbf{5}+2}{3}+\binom{\textbf{1}+2}{3}+\binom{\textbf{1}+2}{3}$\\
\end{tabular}\\
\noindent\begin{tabular}{p{215pt}p{200pt}}
\hspace{-5pt}$14=\binom{\textbf{3}+2}{3}+\binom{\textbf{2}+2}{3}$& $38=\binom{\textbf{5}+2}{3}+\binom{\textbf{1}+2}{3}+\binom{\textbf{1}+2}{3}+\binom{\textbf{2}+1}{3}$\\
\end{tabular}\\
\noindent\begin{tabular}{p{215pt}p{200pt}}
\hspace{-5pt}$15=\binom{\textbf{3}+2}{3}+\binom{\textbf{2}+2}{3}+\binom{\textbf{1}+2}{3}$& $39=\binom{\textbf{5}+2}{3}+\binom{\textbf{2}+2}{3}$\\
\end{tabular}\\
\noindent\begin{tabular}{p{215pt}p{200pt}}
\hspace{-5pt}$16=\binom{\textbf{3}+2}{3}+\binom{\textbf{2}+2}{3}+\binom{\textbf{1}+2}{3}+\binom{\textbf{2}+1}{3}$& $40=\binom{\textbf{5}+2}{3}+\binom{\textbf{2}+2}{3}+\binom{\textbf{1}+2}{3}$\\
\end{tabular}\\
\noindent\begin{tabular}{p{215pt}p{200pt}}
\hspace{-5pt}$17=\binom{\textbf{3}+2}{3}+\binom{\textbf{2}+2}{3}+\binom{\textbf{1}+2}{3}+\binom{\textbf{2}+1}{3}+\binom{\textbf{2}+1}{3}$& $41=\binom{\textbf{5}+2}{3}+\binom{\textbf{2}+2}{3}+\binom{\textbf{1}+2}{3}+\binom{\textbf{2}+1}{3}$\\
\end{tabular}\\
\noindent\begin{tabular}{p{215pt}p{200pt}}
\hspace{-5pt}$18=\binom{\textbf{3}+2}{3}+\binom{\textbf{2}+2}{3}+\binom{\textbf{2}+2}{3}$& $42=\binom{\textbf{5}+2}{3}+\binom{\textbf{2}+2}{3}+\binom{\textbf{1}+2}{3}+\binom{\textbf{2}+1}{3}+\binom{\textbf{2}+1}{3}$\\
\end{tabular}\\
\noindent\begin{tabular}{p{215pt}p{200pt}}
\hspace{-5pt}$19=\binom{\textbf{3}+2}{3}+\binom{\textbf{2}+2}{3}+\binom{\textbf{2}+2}{3}+\binom{\textbf{2}+1}{3}$& $43=\binom{\textbf{5}+2}{3}+\binom{\textbf{2}+2}{3}+\binom{\textbf{2}+2}{3}$\\
\end{tabular}\\
\noindent\begin{tabular}{p{215pt}p{200pt}}
\hspace{-5pt}$20=\binom{\textbf{4}+2}{3}$& $44=\binom{\textbf{5}+2}{3}+\binom{\textbf{2}+2}{3}+\binom{\textbf{2}+2}{3}+\binom{\textbf{2}+1}{3}$\\
\end{tabular}\\
\noindent\begin{tabular}{p{215pt}p{200pt}}
\hspace{-5pt}$21=\binom{\textbf{4}+2}{3}+\binom{\textbf{1}+2}{3}$& $45=\binom{\textbf{5}+2}{3}+\binom{\textbf{3}+2}{3}$\\
\end{tabular}\\
\noindent\begin{tabular}{p{215pt}p{200pt}}
\hspace{-5pt}$22=\binom{\textbf{4}+2}{3}+\binom{\textbf{1}+2}{3}+\binom{\textbf{1}+2}{3}$& $46=\binom{\textbf{5}+2}{3}+\binom{\textbf{3}+2}{3}+\binom{\textbf{1}+2}{3}$\\
\end{tabular}\\
\noindent\begin{tabular}{p{215pt}p{200pt}}
\hspace{-5pt}$23=\binom{\textbf{4}+2}{3}+\binom{\textbf{1}+2}{3}+\binom{\textbf{1}+2}{3}+\binom{\textbf{2}+1}{3}$& $47=\binom{\textbf{5}+2}{3}+\binom{\textbf{3}+2}{3}+\binom{\textbf{1}+2}{3}+\binom{\textbf{2}+1}{3}$\\
\end{tabular}\\
\noindent\begin{tabular}{p{215pt}p{200pt}}
\hspace{-5pt}$24=\binom{\textbf{4}+2}{3}+\binom{\textbf{2}+2}{3}$& $48=\binom{\textbf{5}+2}{3}+\binom{\textbf{3}+2}{3}+\binom{\textbf{1}+2}{3}+\binom{\textbf{2}+1}{3}+\binom{\textbf{2}+1}{3}$
\end{tabular}\\

From this table, we can deduce the number of rows from type~$\ref{type1}$ and type~$\ref{type2}$ modifications by summing up the bold digits. For example, if $\overline{b}_j=23$, then the number of additional rows is $4+1+1+2=8$. To obtain $\overline{b}$ singular square submatrices in $M$, we combine type~$\ref{type1}$ and type~$\ref{type2}$ modifications with the $t_0+t_1+\dotsb+t_{j-1}$ rows from a type~$\ref{type3}$ modification. It remains to verify that the total number of rows we have used is at most $r$. Once again, we employ Mathematica to finish the verification, and the program code is provided in Appendix~\ref{11<=r<=48} for reference.
\end{proof}

\begin{theorem}\label{thmb<=r+3Cr}
Conjecture~$\ref{conj}$ holds under the additional condition that $b\leq\binom{r+3}{3}$, except when $(n,r,b)=(6,3,11)$.
\end{theorem}

\begin{proof}
When $3\leq r\leq10$, we verify through explicit constructions that $(r,3,b)^*$-matrices exist for all integers $1\leq b\leq\binom{r+3}{3}$ except when $(r,3,b)=(3,3,11)$ (see Appendix~\ref{3<=r<=10}). By Lemma~\ref{rkbigger}, for all integers $r\geq11$, $(r,3,b)^*$-matrices exist for all integers $1\leq b\leq\binom{10+3}{3}$.

When $r\geq11$, Propositions~$\ref{Indk=3}$ and $\ref{Basek=3}$ imply that $(r,3,b)^*$-matrices exist for all integers $\binom{r+2}{3}=\binom{r+3}{3}-\binom{r+2}{2}\leq b\leq\binom{r+3}{3}$. Note that $\big[1,\binom{10+3}{3}]\cup\overset{r}{\underset{i=11}{\bigcup}}\big[\binom{i+2}{3},\binom{i+3}{3}\big]=\big[1,\binom{r+3}{3}\big]$. By Lemma~\ref{rkbigger}, for all integers $r\geq11$ and $1\leq b\leq\binom{r+3}{3}$, $(r,3,b)^*$-matrices exist.

Therefore, $(n,r,b)$-matrices exist for all integers $1\leq b\leq\min\left\{\binom{r+3}{3},\binom{n}{r}\right\}$, except when $(n,r,b)=(6,3,11)$.
\end{proof}

Now, we have all the tools for proving Theorem~$\ref{rlarge}$.

\begin{proof}[Proof of Theorem~$\ref{rlarge}$.]
For each $i=0,1,2,\dotsc,k-3$, let $r_0(k-i)\in\N$ be the constant obtained by applying Proposition~$\ref{AsymEst}$ to $k-i$. Without loss of generality, assume that $r_0(3)\leq r_0(4)\leq\dotsb\leq r_0(k)$. Let $R_0=r_0(k)$, and let integers $R_1,R_2,\dotsc,R_{k-3}$ be such that
\begin{enumerate}
\item $R_0\leq R_1\leq R_2\leq\dotsb\leq R_{k-3}$, and
\item $\binom{R_{i}+k-i-1}{k-i}\leq\binom{R_{i+1}+k-(i+1)}{k-(i+1)}$ for all $i=0,1,2,\dotsc,k-4$.
\end{enumerate}

Let $R=R_{k-3}$, and fix $r\geq R$. By Theorem~$\ref{blarge}$, an $(r,k,b)^*$-matrix exists for all integers $\binom{R_0+k-1}{k}\leq b\leq\binom{r+k}{r}$. By Theorem~$\ref{blarge}$ again, for each $i=1,2,\dotsc,k-3$, since $R_i\geq r_0(k-i)$, an $(R_i,k-i,b)^*$-matrix exists for all integers $\binom{R_i+k-i-1}{k-i}\leq b\leq\binom{R_i+k-i}{k-i}$. By Lemma~$\ref{rkbigger}$, for each $i=1,2,\dotsc,k-3$, an $(r,k,b)^*$-matrix exists for all integers $\binom{R_i+k-i-1}{k-i}\leq b\leq\binom{R_i+k-i}{k-i}$. Our definition of $R_i$ implies that $\big[\binom{R_0+k-1}{k},\binom{r+k}{r}\big]\cup\overset{k-3}{\underset{i=1}{\bigcup}}\big[\binom{R_i+k-i-1}{k-i},\binom{R_i+k-i}{k-i}\big]$ covers all integers $\binom{R+2}{3}\leq b\leq\binom{r+k}{r}$. Therefore, an $(r,k,b)^*$-matrix exists for all integers $\binom{R+2}{3}\leq b\leq\binom{r+k}{k}$.

Finally, we are done by Theorem~$\ref{thmb<=r+3Cr}$, which says an $(r,k,b)^*$-matrix exists for all integers $1\leq b\leq\binom{R+3}{3}$.
\end{proof}

\section{Conclusion and remarks}
Throughout this paper, the base field is $\C$, but the same argument works for any algebraically closed field. Moreover, if we consider the algebraic closure $\overline{\F_p}$ for some prime $p$, the constructed $(r,k,b)^*$-matrix naturally descends to some finite extension of $\F_p$. However, we cannot ensure that there is a fixed finite extension which captures all of them. In any case, the matroid structure arising from these matrices will not be affected.

All our current results focus on the situation when $n-r$ is small comparing with $r$. Recall from Proposition~$\ref{k<=r}$ that we only need to consider $r\leq n\leq2r$. Hence, our next goal is to investigate the case when $n$ is close to $2r$. In view of the non-existence of $(6,3,11)$-matroids, this latter goal should be much harder. Nevertheless, we believe that the general direction towards a complete solution to Conjecture~\ref{conj} will be another asymptotic result concerning the existence of $(r,k,b)^*$-matrices for $k$ closer to $r$, which should isolate a finite number of cases for direct checking.

\appendix
\section{Appendix: Mathematica code for verifying base cases in Proposition~$\ref{Indk=3}$}\label{49<=r<=203}

We use Mathematica to check the validity of Proposition~\ref{Indk=3} for $49\leq r\leq 203$.\\
\\
\texttt{truthvalue = True;\\
Do[a = Table[0, \{s, r\}];\\
\indent Do[btemp = bbar;\\
\indent\indent Do[a[[s]] = Floor[btemp/Binomial[s, 3]];\\
\indent\indent\indent btemp = Mod[btemp, Binomial[s, 3]], \{s, r, 3, -1\}];\\
\indent\indent If[Sum[a[[s]]*s, \{s, 3, r\}] > r, truthvalue = False],\\
\indent\indent \{bbar, 0, Binomial[r + 2, 2]\}],\\
\indent \{r, 49, 203\}];\\
truthvalue}\\

As the final output is true, we finish our verification.

\section{Appendix: Mathematica code for verifying base cases in Proposition~$\ref{Basek=3}$}\label{11<=r<=48}

From the list in Proposition~$\ref{Basek=3}$, we use \texttt{length} below to record the number of rows from type~$\ref{type1}$ and type~$\ref{type2}$ modifications to obtain additional singular submatrices.\\
\\
\texttt{truthvalue = True;\\
length = \{0, 1, 2, 3, 2, 3, 4, 6, 4, 5, 3, 4, 5, 7, 5, 6, 8, 10, 7, 9,\\
\indent 4, 5, 6, 8, 6, 7, 9, 11, 8, 10, 7, 8, 10, 12, 9, 5, 6, 7, 9, 7, 8,\\
\indent 10, 12, 9, 11, 8, 9, 11, 13\};\\
Do[tbinomial = Table[Binomial[t, 3] + Binomial[t, 2] (r - t + 3), \{t, r\}];\\
\indent a = Table[0, \{t, r\}];\\
\indent Do[btemp = bbar;\\
\indent\indent Do[a[[t]] = Floor[btemp/tbinomial[[t]]];\\
\indent\indent\indent btemp = Mod[btemp, tbinomial[[t]]], \{t, r, 2, -1\}];\\
\indent\indent If[Sum[a[[t]]*t, \{t, 2, r\}] + length[[btemp + 1]] > r,\\
\indent\indent\indent truthvalue = False],\\
\indent\indent \{bbar, 0, Binomial[r + 2, 2]\}],\\
\indent \{r, 11, 48\}];\\
truthvalue}\\

As the final output is true, we finish our verification.

\section{Appendix: Constructions of $(r,3,b)^*$-matrices for $3\leq r\leq10$ in Theorem~$\ref{thmb<=r+3Cr}$}\label{3<=r<=10}

When $r=3$, Corollary~$\ref{corb<=r+2Cr}$ implies that $(3,3,b)^*$-matrices exist for all integers $1\leq b\leq\binom{3+2}{2}=10$, and the following constructions produce $(3,3,b)^*$-matrices $M$ for all integers $12\leq b\leq\binom{3+3}{3}=20$.

\begin{center}
\footnotesize\begin{tabular}{|c|c|c|c|c|}
\hline
$b=$& $12$& $13$& $14$& $15$\\
\hline
& & & &\\
$M=$& \begin{tabular}{|c|c|c|}
\hline
$0$& $0$& $1$\\
$1$& $1$& $1$\\
$1$& $1$& $1$\\
\hline
\end{tabular}& \begin{tabular}{|c|c|c|}
\hline
$0$& $0$& $1$\\
$1$& $0$& $1$\\
$1$& $1$& $0$\\
\hline
\end{tabular}& \begin{tabular}{|c|c|c|}
\hline
$0$& $0$& $1$\\
$1$& $0$& $2$\\
$1$& $1$& $1$\\
\hline
\end{tabular}& \begin{tabular}{|c|c|c|}
\hline
$0$& $1$& $1$\\
$0$& $1$& $2$\\
$1$& $1$& $2$\\
\hline
\end{tabular}\\
& & & &\\
\hline
\end{tabular}
\end{center}
\begin{center}
\footnotesize\begin{tabular}{|c|c|c|c|c|c|}
\hline
$b=$& $16$& $17$& $18$& $19$& $20$\\
\hline
& & & & &\\
$M=$& \begin{tabular}{|c|c|c|}
\hline
$0$& $1$& $1$\\
$0$& $1$& $2$\\
$1$& $1$& $3$\\
\hline
\end{tabular}& \begin{tabular}{|c|c|c|}
\hline
$0$& $1$& $1$\\
$1$& $0$& $1$\\
$1$& $1$& $0$\\
\hline
\end{tabular}& \begin{tabular}{|c|c|c|}
\hline
$0$& $1$& $1$\\
$1$& $0$& $1$\\
$1$& $1$& $3$\\
\hline
\end{tabular}& \begin{tabular}{|c|c|c|}
\hline
$0$& $1$& $1$\\
$1$& $1$& $2$\\
$1$& $2$& $5$\\
\hline
\end{tabular}& \begin{tabular}{|c|c|c|}
\hline
$1$& $1$& $1$\\
$1$& $2$& $3$\\
$1$& $3$& $6$\\
\hline
\end{tabular}\\
& & & & &\\
\hline
\end{tabular}
\end{center}

Here is the Mathematica code for counting the number of invertible submatrices of $M$. The technique is to count the number of invertible $r\times r$ submatrices of $A^\top$, which is obtained by stacking the identity matrix $I_r$ on top of $M^\top$.\\
\\
\texttt{b[matrix\_] := Block[\{r, k, choice, listofdet\},\\
\indent \{r, k\} = Dimensions[matrix];\\
\indent choice = Subsets[Table[i, \{i, r + k\}], \{r\}];\\
\indent listofdet = Table[\\
\indent\indent Det[ Join[ IdentityMatrix[r], Transpose[matrix] ] [[ choice[[i]] ]] ],\\
\indent\indent \{i, Binomial[r + k, k]\}];\\
\indent Count[listofdet, u\_ /; u != 0]]}\\

To verify the above constructions of $M$, one may use the following Mathematica code.\\
\\
\texttt{Map[b, \{\{\{0, 0, 1\}, \{1, 1, 1\}, \{1, 1, 1\}\}, \{\{0, 0, 1\}, \{1, 0, 1\}, \{1, 1, 0\}\},\\
\indent \{\{0, 0, 1\}, \{1, 0, 2\}, \{1, 1, 1\}\}, \{\{0, 1, 1\}, \{0, 1, 2\}, \{1, 1, 2\}\},\\
\indent \{\{0, 1, 1\}, \{0, 1, 2\}, \{1, 1, 3\}\}, \{\{0, 1, 1\}, \{1, 0, 1\}, \{1, 1, 0\}\},\\
\indent \{\{0, 1, 1\}, \{1, 0, 1\}, \{1, 1, 3\}\}, \{\{0, 1, 1\}, \{1, 1, 2\}, \{1, 2, 5\}\},\\
\indent \{\{1, 1, 1\}, \{1, 2, 3\}, \{1, 3, 6\}\}\}]}\\

Corollary~\ref{corb<=r+2Cr} also implies that a $(4,3,11)^*$-matrix exists since $11\leq\binom{4+2}{2}=15$. The following constructions produce $(r,3,b)^*$-matrices $M$ for all integers $4\leq r\leq10$ and $\binom{r+2}{3}<b\leq\binom{r+3}{3}$. Note that $[1,10]\cup[12,20]\cup\{11\}\cup\overset{r}{\underset{i=4}{\bigcup}}\big[\binom{i+2}{3},\binom{i+3}{3}\big]=\big[1,\binom{r+3}{3}\big]$. By Proposition~\ref{rkbigger}, for all integers $3\leq r\leq10$ and $1\leq b\leq\binom{r+3}{3}$, $(r,3,b)^*$-matrices exist except $(r,3,b)=(3,3,11)$. For the Mathematica code, one may refer to the ancillary files, or download it from \url{http://faculty.kutztown.edu/wong/MathematicaCode(Theorem4.4).txt}.

When $r=4$ and $\binom{4+2}{3}=20<b\leq\binom{4+3}{3}=35$, the following constructions produce $(4,3,b)^*$-matrices.

\begin{center}
\footnotesize

\end{center}

When $r=5$ and $\binom{5+2}{3}=35<b\leq\binom{5+3}{3}=56$, the following constructions produce $(5,3,b)^*$-matrices.

\begin{center}
\footnotesize%
\end{center}

When $r=6$ and $\binom{6+2}{3}=56<b\leq\binom{6+3}{3}=84$, the following constructions produce $(6,3,b)^*$-matrices.

\begin{center}
\footnotesize%
\end{center}

When $r=7$ and $\binom{7+2}{3}=84<b\leq\binom{7+3}{3}=120$, the following constructions produce $(7,3,b)^*$-matrices.

\begin{center}
\footnotesize%
\\
& & & & & & &\\
\hline
\end{tabular}
\end{center}

\end{document}